\documentclass[12pt,reqno]{amsart}
\usepackage{amsfonts,amsmath,amssymb,epsf,color,array,colortbl,amscd,bbm,tensor,tikz,comment,caption,soul, enumitem}
\usepackage{tikz}
\usetikzlibrary{shadings}
\usetikzlibrary{matrix,arrows}
\usepackage{fullpage}
\theoremstyle{plain}
\newtheorem{theorem}{Theorem}
\newtheorem{proposition}[theorem]{Proposition}

\newtheorem{corollary}[theorem]{Corollary}
\newtheorem{lemma}[theorem]{Lemma}
\theoremstyle{definition}

\newtheorem{remark}[theorem]{Remark}

\allowdisplaybreaks

\newcommand {\Z}{\mathbb{Z}}

\newcommand {\CC}{\mathbb{C}}

\newcommand{\be}{\begin{equation}}
\newcommand{\ee}{\end{equation}}

\newcommand{\qdim}[1]{\mathrm{qdim}[#1]}

\newcommand{\Proj}{\mathcal P}

\newcommand{\Ck}{\mathbb{C}_{kr}^H}
\newcommand{\U}{{\overline{U}_q^H(\mathfrak{sl}_2)}}
\newcommand{\UR}{\overline{U}_q(\mathfrak{sl}_2)}
\newcommand{\UU}{{U}_q(\mathfrak{sl}_2)}
\newcommand{\w}{\textsf{w}}

\newcommand{\M}{\mathcal M(r)}
\newcommand{\strip}[1]{\mathbb{S}(#1)} 
\newcommand{\im}{\text{Im}} 
\newcommand{\re}{\text{Re}} 

\renewenvironment{proof}[1][Proof]{\begin{trivlist} \item[\hskip \labelsep {\bfseries #1:}]}{\qed\end{trivlist}}

\begin{document}

\title{Logarithmic Link Invariants of {\small $\U$} and\\  Asymptotic Dimensions of Singlet Vertex Algebras}

\author{Thomas Creutzig, Antun Milas and Matt Rupert } 
\thanks{T.C. was supported by an NSERC Research Grant (RES0020460)}
\thanks{  A.M. was  supported by a Simons Foundation Collaboration Grant ($\#$ 317908)}
\maketitle

\begin{abstract}
{We study relationships between the restricted unrolled quantum group $\U$ at $2r$-th root of unity $q=e^{\pi i/r}, r \geq 2$, and the singlet vertex operator algebra $\M$. }
We use deformable families of modules to efficiently compute $(1, 1)$-tangle invariants colored with projective modules of { $\U$}. These relate to the colored Alexander tangle invariants studied  in \cite{ADO,M1}. It follows that the regularized asymptotic dimensions of characters of $\M$ coincide with the corresponding modified traces of open Hopf link invariants.  We also discuss various categorical properties of $\M$-mod in connection to braided tensor categories.
\end{abstract}

\section{Introduction}

This work is concerned with the restricted "unrolled" quantum group $\U$ at $2r$-th root of unity $q=e^{\pi i/r}, r \geq 2$, and the singlet vertex operator algebra $\M$. Representation categories of both are neither semi-simple nor do they have finitely many simple objects. While this quantum group has been used to construct link and $3$-manifold invariants \cite{BCGP,CGP,CGP2,CMu,GP,GPT,GPT2},  categorical properties of the singlet vertex operator algebra (and more generally, vertex algebras with infinitely many simple objects) are still poorly understood. Previously, it was realized  in examples that modular-like properties of characters \cite{AC, CM, CR1, CR2, CR3}, as well as their {\em asymptotic} dimensions (often also called {\em quantum} dimensions) \cite{CM, CMW}, relate to the fusion ring of the singlet vertex algebras and other vertex (super)algebras.
This begs for a categorical interpretation and in this work the relation of $\M$ and $\U$ gives such an interpretation via open Hopf link invariants. 
In particular, this shows that the Jacobi variable introduced in \cite{CM} as a regularization parameter for the classical false theta functions 
has a novel interpretation from the point of view of quantum topology. 

\subsection{Tensor Categories and Vertex Algebras}

Vertex operator algebras are important sources of braided and modular tensor categories.  If a vertex operator algebra $V$ is regular (i.e. $C_2$-cofinite and rational), together with some additional mild conditions, then it is well-understood \cite{H1, H2} that its representation category is modular and especially ribbon. Moreover in this case, there are three actions of the modular group: on the linear span of torus one-point functions, a categorical one given by twists and Hopf links and the one that diagonalizes the fusion rules (this is only a $S$-matrix). 
All three coincide in the appropriate sense.

If $V$ is not rational but still $C_2$-cofinite (satisfying a few additional assumptions), then there is still a modular group action on the space of one-point functions on the torus \cite{Mi}. Moreover, there is also a modular group action in the category provided it is ribbon \cite{Ly1, Ly2} and a relation between the character $S$-matrix and logarithmic Hopf link invariants has been given in \cite{CG,CG2} for the triplet vertex algebras. Here we would like to extend these observations beyond categories with finitely many simple objects. From the vertex algebra point of view, vertex algebras which are not $C_2$-cofinite are considerably difficult to study. One issue with these algebras is that the category of weak modules  is 
way too big and only after restriction to a subcategory we hope to have good categorical properties. In an important series of papers, Huang, Lepowsky and Zhang 
\cite{HLZ} obtained sufficient conditions on a subcategory ${C}$ to posses a braided (vertex) tensor category structure. Roughly speaking, they proved that if all objects in an abelian subcategory ${C}$ of generalized $V$-modules satisfy the $C_1$-cofiniteness condition, the category is closed under $P(z)$-tensor products, and a few additional conditions, then the convergence and extension properties for products and iterates hold in $C$, and thus the category 
can be endowed with a braided tensor category structure. More precisely,  Assumptions 10.1, 12.1 and 12.2 \cite{HLZ} have to be satisfied. The most difficult part in the verification of these axioms is that a suitable subcategory is closed under the $P(z)$-tensor product. We should mention that Miyamoto recently obtained a sufficient condition on the closure of the tensor product in a $C_1$-cofinite category of modules \cite{Mi1}.  But his result alone does not give a braided tensor product. 

The singlet vertex algebra $\M$, $r \in \mathbb{N}_{\geq 2}$, \footnote{Regretfully, there is no consistent notation for the singlet algebra in the literature; previously it was 
also denoted by $\mathcal{W}(2,2r-1)$ in \cite{CM} as well as by $\overline{M(1)}$ in \cite{AD,AM1}.} is a prominent example among irrational non $C_2$-cofinite vertex algebras and it was studied by many authors \cite{AD, AM1,AM2,AM3,CM, CMW}.  This vertex algebra (subalgebra of the triplet vertex algebra) contains both atypical and typical representations so it serves as the best testing ground for categorical exploration beyond $C_2$-cofiniteness. We have already understood in previous works that asymptotic dimensions of characters relate to representations of the fusion ring \cite{CM, CMW} and our conjecture was that this has a precise categorical meaning.  Here, we will give such an interpretation. It is believed that categories of $\U$-modules and the singlet VOA $\M$-modules are equivalent as monoidal categories and maybe even as braided tensor categories, after  
restriction to suitable sub-categories. This restriction is indeed needed  in order to have a braiding on the quantum group side - one needs the category of {\em weight modules} introduced in \cite{CGP}.  

\subsection{Summary of the present work}

Representations of the tensor ring in a ribbon category are directly given by open logarithmic Hopf link invariants \cite{T} (for a proof in non-strict categories see \cite{CG}). Here, we first successfully compare them with the asymptotic dimensions of characters and secondly we find a novel way of computing them. Previously, they have been computed using the known tensor ring \cite{CGP}, which from our perspective is not ideal as we are seeking ways to better understand the still inaccessible fusion ring of VOAs. Our computation is a deformation argument analogous to ideas of Murakami and Nagatomo \cite{M2, MN} in the case of the restricted quantum group $\UR$ but also motivated by the idea of deformable families of VOAs \cite{CL}. The idea is that if one has a continuous set of modules, all but a discrete set of them semi-simple, then one can construct a deformable family of modules $M(x)$ which specializes to specific modules if specializing the variable $x$ and especially at non-generic position also specializing to the indecomposable but reducible modules of particular interest in logarithmic conformal field theory. Moreover, this process commutes with the computation of invariants of interest such as Hopf links and twists. Since these are next to trivial to compute on simple modules our process gives a nice way of obtaining them also for the complicated cases of indecomposable but reducible modules. 

Our strategy of computation actually extends straight forwardly to any open $(1, 1)$ tangle invariants. It turns out that the results very nicely compare to the colored Alexander invariants introduced by Akutsu, Deguchi and Ohtsuki in \cite{ADO} and further developed by Murakami \cite{M1, MN, M2}. This has already been observed earlier if only simple projective modules were used as colors \cite{GPT}. We would like to stress one important difference between $\U$ and $\UR$. The latter is not braidable and therefore the category of finite-dimensional $\UR$-modules does not have a universal $R$-matrix \cite{KS}.  On the other hand the category of (suitably defined) weight modules for $\U$ is braided. 
We also would like to announce that this problem for $\UR$ can be cured by finding a suitable non-trivial associator in the module category of $\UR$ \cite{GR, CGR}. This associator is found in identifying the module category of $\UR$ as the representation category of local modules of an algebra in $\U$ \cite{CGR}. It remains however to be proven that the resulting module category gives rise to the colored Alexander invariants studied by Murakami.

\subsection{Results}

It is believed that the representation categories of $\U$ and $\M$ are equivalent as monoidal and hopefully as braided tensor categories (we will have a more precise conjecture below). The fusion ring of $\M$ is not known, its Grothendieck ring has been conjectured in \cite{CM}. 
Comparison does work though (see Section \ref{sec:comparison}):
\begin{theorem}
{Irreducible} representations of $\U$ and $\M$ are related  as follows:
\begin{enumerate}
\item Assume that the Grothendieck ring of $\M$ is as conjectured (see Section  \ref{sec:singlet} ).
Let $\alpha \in (\mathbb{C} \setminus \mathbb{Z}) \cup r \mathbb{Z}$. Then the map $\varphi :V_{\alpha} \mapsto F_{\frac{\alpha+r-1}{\sqrt{2r}}}$, $S_i \otimes \Ck \mapsto M_{1-k,i+1}$ between simple objects for $\U$ and $\M$ is bijective and induces a ring isomorphism. 
\item 
Let $k\in\mathbb Z$ and $j\in \{0, \dots, r-2\}$ and let $\epsilon \in \mathbb C$ satisfy $\re(\epsilon)<B_\epsilon^r$ as well as $\epsilon$ in $\mathbb S(k, j+1+r(k+1))$ (for the precise definitions of these sets, see Section \ref{sec:singlet}) then
\[
\text{\em qdim}[\varphi(X)^{\epsilon}]=\frac{t_{P_j \otimes \Ck}(\Phi_{ X ,P_j \otimes \Ck} \circ x_{j,k})}{t_{P_j \otimes \Ck}(\Phi_{S_0,P_j \otimes \Ck} \circ x_{j,k})}, 
\]
and if $\re(\epsilon)>B_\epsilon^r$ then for $\alpha = i\sqrt{2r}\epsilon$ we have
\[
\text{\em qdim}[\varphi(X)^{\frac{-i \alpha}{\sqrt{2r}}}]=\frac{t_{V_{\alpha}}(\Phi_{X,V_{\alpha}})}{t_{V_{\alpha}}(\Phi_{S_0,V_{\alpha}})}.
\]
Here $t_{X}$ is the modified trace on the ideal of negligible (projective) objects $X$.
\end{enumerate}

\end{theorem}
If the representation categories of $\U$ and $\M$ are braided equivalent then this means that the asymptotic dimensions of regularized characters of the VOA $\M$ very nicely capture the modified traces of the logarithmic Hopf link invariants. In the continuous regularization regime these are the ordinary modified traces while the stripwise constant regime corresponds to traces of logarithmic Hopf link invariants weighted with the nilpotent endomorphisms. 

In \cite{CM}, the regularized quantum dimensions of modules were used to at least conjecture the Grothendieck ring of $\M$ and in \cite{CGP} the Hopf link invariants of $\U$ were computed using the known tensor ring of $\U$. Our picture is that one should use the Hopf link invariants to compute the tensor ring of $\U$ and we indeed can find a strategy that works well. 
The key Lemma (see Theorem \ref{thm:deform}) is the construction of a deformable family of modules $X_\epsilon$ for $\epsilon$ in $(-\frac{1}{2}, \frac{1}{2})$ such that 
\[
X_{\epsilon}=\left\{ \begin{array}{cc}
V_{1+i-r+\ell r+\epsilon} \oplus V_{-1-i+r+\ell r+\epsilon} & \text{ if }\epsilon \neq 0\\
P_i  \otimes \CC_{\ell r}^H& \text{ if } \epsilon = 0\\
\end{array} \right.\\
\]
This deformable family can be used to compute open Hopf link invariants and more generally open tangle invariants of $(1, 1)$ tangles $T$.
For this we recall Section 3 of \cite{GPT}. Let $L$ be a $\mathcal C$-colored ribbon graph, such that at least one of the colors is a simple $V_\lambda$. 
Let $T_\lambda$ be the colored $(1, 1)$-ribbon graph by cutting the edge belonging to $V_\lambda$. Then the re-normalized Reshetikhin-Turaev link invariant is
\[
F'(L) = t_{V_{\lambda}}(T_\lambda),
\]
where $ t_{V_{\lambda}}$ is the modified trace on $V_\lambda$. 
These invariants where shown in \cite{GPT} to coincide with Murakami's Alexander invariants \cite{M1}. We can now extend these results to any projective module. For this let $L$ be as above but with at least one of the colors the module $P_i\otimes  \CC_{\ell r}^H$. Let $T_{P_i\otimes  \CC_{\ell r}^H}$ be the colored $(1, 1)$-ribbon graph by cutting the edge belonging to $P_i\otimes  \CC_{\ell r}^H$. Let $T_\lambda$ be the colored $(1, 1)$-ribbon graph obtained by replacing the open strand (which is colored with $P_i\otimes  \CC_{\ell r}^H$) with $V_\lambda$. Let $\textbf{d}(X)=t_X(Id_X)$ be the modified dimension of a projective module $X$ and let $x_{i, \ell r}$ be the nilpotent endomorphism of $P_i\otimes  \CC_{\ell r}^H$ normalized as in section \ref{sec:proj}, then (Theorem \ref{thm:links}):
\begin{theorem}
The colored $(1, 1)$-ribbon graph $T_{P_i\otimes  \CC_{\ell r}^H}$ satisfies
\begin{equation}\nonumber
\begin{split}
 t_{P_i\otimes  \CC_{\ell r}^H}\left(T_{P_i\otimes  \CC_{\ell r}^H}\right) &= \lim\limits_{\epsilon \to 0}\left( t_{V_{1+i-r+\ell r+\epsilon}}\left(T_{1+i-r+\ell r+\epsilon}\right) +   t_{V_{-1-i-r+\ell r+\epsilon}}\left(T_{-1-i-r+\ell r+\epsilon}\right) \right)
\end{split}
\end{equation}
and
 \begin{equation}\nonumber
\begin{split}
T_{P_i\otimes  \CC_{\ell r}^H} &=  a {\rm Id}_{P_i \otimes  \CC_{\ell r}^H} +bx_{i, \ell r}
\end{split}
\end{equation}
with coefficients
 \begin{equation}\nonumber
\begin{split} 
a &=  \lim\limits_{\epsilon \to 0} \frac{t_{V_{-1-i-r+\ell r+\epsilon}}\left(T_{-1-i-r+\ell r+\epsilon}\right)}{\textbf{d}\left(T_{-1-i+r+\ell r+\epsilon}\right)} =  \lim\limits_{\epsilon \to 0}  \frac{t_{V_{1+i-r+\ell r+\epsilon}}\left(T_{1+i-r+\ell r+\epsilon}\right)}{\textbf{d}\left(V_{1+i-r+\ell r+\epsilon} \right)}
\qquad\text{and} \\
b &=  \frac{r}{2\pi i \{1+i\}} \left(\frac{d}{d\lambda} \frac{t_{V_{\lambda}}\left(T_{\lambda}\right)}{\textbf{d}\left(T_{\lambda}\right)} \Big|_{\lambda=\ell r+r-i-1} 
 - \frac{d}{d\lambda}\frac{t_{V_{\lambda}}\left(T_{\lambda}\right)}{\textbf{d}\left(V_{\lambda} \right)}\Big|_{\lambda=i+i-r+\ell r} \right). 
\end{split}
\end{equation}
\end{theorem}
The construction of the deformable family of modules $X_\epsilon$ is the only key ingredient in proving this theorem. Murakami and Nagatomo \cite{MN} also constructed a deformable family of modules of the semi-restricted quantum group that specialized at a special point to a projective indecomposable but reducible module and that only at this special point also became a module of $\UR$, while in \cite{M1}, Murakami varied the $q$-parameter of $\UU$ such that in the limit $q=e^{\frac{2\pi i}{2r}}$ the module specialized to projective indecomposable but reducible module. 
In both cases it was then used that the employed $R$-matrix of the involved quantum groups is basically the same. We however recall again, that $\UR$ is not braided. 

Presently, the only vertex operator algebras with infinitely many simple objects and non semi-simple representations whose categories of modules are known to be braided is the Heisenberg vertex algebra (see Remark \ref{last}) and the Kazhdan-Lusztig category \cite{KL}.  
The last section present an attempt to push the tensor product theory beyond the Heisenberg vertex algebra. We think that the singlet algebra is an excellent candidate in this 
direction. In order to apply the Huang-Lepowsky-Zhang theory, several highly non-trivial conditions have to be verified including the $C_1$-cofiniteness of a suitable subcategory of modules. Although at this stage we do not have a full proof that 
$\mathcal{M}(r)$-mod is braided, we reduced this problem to a purely representation theoretic condition (assumptions (a) and (b) below):
\begin{theorem} All finite-length $\mathcal{M}(r)$-modules are $C_1$-cofinite. 
Assume that (a) every $C_1$-cofinite, $\mathbb{N}$-graded module is of finite length, and (b) every finitely generated, generalized $\mathbb{N}$-graded $\mathcal{M}(r)$-module is $C_1$-cofinite. Then, the category of $\mathbb N$-graded, $C_1$-cofinite $\mathcal M(r)$-modules can be equipped with a braided tensor  category structure. \end{theorem}

\subsection{Future work}
{This work has several ramifications and extensions. In \cite{CM2}, we introduced and studied regularized characters of modules of certain non $C_2$-cofinite  $W$-algebras denoted by ${W}^0(Q)_r$, where $r \geq 2$ and $Q$ is the root lattice of a simply-laced simple Lie algebra $\frak{g}$ . These vertex algebras are 
"higher rank" generalization of the singlet vertex algebra $\mathcal{M}(r)$. Their categories of modules are expected to be closely related to the category of modules for higher rank unrolled quantum groups $\overline{U}^{\frak{h}}_q(\frak{g})$ at $2r$-th root of unity.
We plan to study asymptotic dimensions of ${W}^0(Q)_r$-modules in connection to quantum invariants of knots and links colored with representations 
of $\overline{U}^{\frak{h}}_q(\frak{g})$. }

{\bf Acknowledgements.} 
T.C. is very grateful to Terry Gannon for collaboration on $C_2$-cofinite VOAs and modular tensor categories resulting in the works \cite{CG,CG2} where the importance of logarithmic Hopf links in $C_2$-cofinite VOAs has been realized. He is also thankful to Azat Gainutdinov and Ingo Runkel for valuable discussions on the relation of $\U$ and $\UR$ \cite{CGR}.  A.M. would like to thank Yi-Zhi Huang for discussions on various aspects of the vertex tensor product theory \cite{HLZ} and to Dra\v{z}en Adamovi\'c. Finally, we are grateful to Jun Murakami for a correspondence.

\section{The singlet vertex operator algebra $\M$}\label{sec:singlet}

Let $r\in \mathbb N_{\geq 2}$.
Here, we review necessary information of the vertex algebra $\M$ following \cite{CM}; see also \cite{AD, AM1, AM2}.
The vertex algebra $\M$ is realized as a subalgebra of the rank one Heisenberg vertex algebra $F_0$.
It is strongly generated by the Virasoro vector $\omega$ (suitably chosen such that the central charge is $1-\frac{6(p-1)^2}{p}$),
together with one primary field of conformal weight $2r-1$, usually denoted by $H$. 
Introduce $$\alpha_+=\sqrt{2r}, \  \alpha_-=-\sqrt{2/r}, \  {\rm and} \   \alpha_0=\alpha_-+\alpha_+.$$
Denote by $F_{\lambda}$ the Fock space with the highest weight $\lambda \in \mathbb{C}$, which is also an $F_0$-module.
All irreducible $\mathcal{M}$-modules are realized as subquotients of $F_\lambda$, which we now describe briefly.

Let $L=\sqrt{2r} \mathbb Z$, viewed as a rank one lattice, and $L'$ denotes its dual lattice. 
Then {\em typical} simple modules $F_\lambda$ are parameterized by $\lambda$ in $\left(\mathbb C \setminus L'\right) \cup L$, while atypical simple modules $M_{t, s} \subset F_{\alpha_{t,s}}$ are parameterized by integers $t, s$ with $1\leq s \leq r-1$. 
Characters of irreducible $\mathcal{M}(r)$-modules can be 
easily computed. Irreducible characters admit an $\epsilon$-regularization, as explained in \cite{CM},
where they are denoted by ${\rm ch}[X^\epsilon]$, where $X$ is an $\mathcal{M}(r)$-module. Here $\epsilon \in \mathbb{C}$.
The main result of \cite{CM} is a formula for the modular transformation of regularized partial and false theta functions, which then gives modular properties of regularized characters. These in turn give a Verlinde-type algebra for the regularized characters, where the product is defined (for ${\rm Re}(\epsilon)>0$\footnote{In \cite{CM}, we used $-\epsilon$ instead of $\epsilon$ so the formula was given in the ${\rm Re}(\epsilon)<0$ region.}) as 
$${\rm ch}[X_a^\epsilon] \times {\rm ch}[X_b^\epsilon]=\int_{\mathbb{R}} \int_{\mathbb{R}} \frac{S_{a \rho}^\epsilon S_{b \rho}^\epsilon \overline{S_{\rho \mu}^{-\epsilon}}}{S_{(1,1) \rho}^\epsilon} {\rm ch}[F_{\mu}^\epsilon] d \mu d \rho,$$ and where $S_{ \cdot, \cdot}$ defines the $S$-kernel.
For irreducible characters this formula reads
\begin{equation}\label{CM-fusion}
\begin{split}
\mathrm{ch}[F_\lambda^\epsilon]\times  \mathrm{ch}[F_\mu^\epsilon] &= \sum_{\ell=0}^{p-1}\mathrm{ch}[F_{\lambda+\mu+\ell\alpha_-}^\epsilon], \qquad\quad
\mathrm{ch}[M_{t ,s}^\epsilon]\times  \mathrm{ch}[F_\mu^\epsilon] = 
\sum_{\substack{\ell=-s+2\\ \ell+s=0\, \mathrm{mod}\, 2}}^{s}\mathrm{ch}[F_{\mu+\alpha_{r,\ell}}^\epsilon]\\
\mathrm{ch}[M_{t ,s}^\epsilon]\times  \mathrm{ch}[M_{t',s'}^\epsilon] &= 
\quad\sum_{\substack{\ell=|s-s'|+1\\ \ell+s+s'=1\, \mathrm{mod}\, 2}}^{min \{ s+s'-1,p \}}\mathrm{ch}[M_{t+t'-1,\ell}^\epsilon] \\
& +\sum_{\substack{\ell=p+1\\ \ell+s+s'=1\, \mathrm{mod}\, 2}}^{s+s'-1}\Bigl(\mathrm{ch}[M_{t+t'-2,\ell-p}^\epsilon]+
\mathrm{ch}[M_{t+t'-1,2p-\ell}^\epsilon]+\mathrm{ch}[M_{t+t',\ell-p}^\epsilon]\Bigr).
\end{split}
\end{equation}

Then regularized asymptotic dimensions are introduced as
\begin{equation}
 \qdim{X^\epsilon} := \lim_{\tau\rightarrow 0+}\frac{\text{ch}[X^\epsilon(\tau)]}{\text{ch}[M_{1,1}^\epsilon](\tau)}.
\end{equation}
Following \cite{CMW} introduce
\begin{equation}\label{eq:regime}
B_\epsilon^r:= - \mathrm{min}\left\{ \Big|\frac{m}{\sqrt{2r}}-\mathrm{Im}\left(\epsilon\right)\Big|\  \Big\vert \ m\in\mathbb Z\setminus r\mathbb Z\ \right\}. 
\end{equation}
Then for ${\rm Re}(\epsilon)>B_\epsilon^r$ the regularized asymptotic dimensions are
\begin{equation}
\begin{split}
\qdim{F_\lambda^\epsilon}&= q_{\epsilon}^{2\lambda-\alpha_0}\frac{\mathrm{sin}(-\pi\alpha_+\epsilon i)}{\mathrm{sin}(\pi\alpha_-\epsilon i)}=
q_{\epsilon}^{2\lambda-\alpha_0}\sum_{\substack{\ell=-p+1\\ \ell+p=1\,\mathrm{mod}\, 2}}^{p-1}  q_{\epsilon}^{\alpha_-\ell}\\
\qdim{M_{t,s}^\epsilon}&= q_{\epsilon}^{-(t-1)\alpha_+}\frac{\mathrm{sin}(\pi s\alpha_-\epsilon i)}{\mathrm{sin}(\pi\alpha_-\epsilon i)}=
q_{\epsilon}^{-(t-1)\alpha_+}\sum_{\substack{\ell=-s+1\\ \ell+s=1\,\mathrm{mod}\, 2}}^{s-1}  q_{\epsilon}^{\alpha_-\ell}
\end{split}
\end{equation}
and for $\re(\epsilon)<B_\epsilon^r$ the answer is $\qdim{F_\lambda^\epsilon}=0$ and
for \(\epsilon\in\strip{k,m}, k\in\mathbb{Z},m=0\dots,2r-1\),
\begin{equation} \label{main-qdim}
\qdim{M_{t ,s}^{\epsilon}}=
\begin{cases}
  (-1)^{m(t-1)} \displaystyle{\frac{\sin(\pi m s/r)}{\sin(\pi m /r)}}&\text{if } m\neq0, r,\\
  (-1)^{(m+1)(t-1)+\frac{m}{r}(s-1)} \displaystyle{\frac{\sin(\pi s/r)}{\sin(\pi
    /r)}}&\text{if } m=0, r.
\end{cases}
\end{equation}
with
\begin{displaymath}
  \strip{k,m}=\left\{\epsilon\in\mathbb{C}\left| k+\frac{2m-1}{4r}<
      \frac{\im(\epsilon)}{\sqrt{2r}} 
      < k+\frac{2m+1}{4r} \right.\right\},
\end{displaymath} 
$k\in\mathbb{Z}$ where $m=0,\dots,2r-1$.
It turns out that the algebra of quantum dimensions and above Verlinde algebra of characters coincide. Further the conjecture of \cite{CM} is that these relations also hold in the Grothendieck ring of the module category of $\M$. Note, that the latter is also only conjectured to be braided.  

\section{The unrolled restricted quantum group $\U$}

Throughout $q=e^{\pi i /r}$, $r \geq 2$. Also, $\{\alpha \}=q^{\alpha}-q^{-\alpha}$, $[\alpha]=\frac{\{ \alpha \}}{\{ 1 \}}$, and $[n]!=[n] \cdots [1]$.

In this section we review basic facts about the unrolled quantum group $\U$ following primarily  \cite{CGP}.
The quantum group $\U$ as a unital associative algebra is generated $E$,$F$, $K$, $K^{-1}$ and $H$, with the following relations:
$$KE=q^{2} EK , \ \ KF=q^{-2} FK,   \ \ HK=KH $$
$$[H,E]=2E, \ \ [H,F]=-2F, \ \ [E,F]=\frac{K-K^{-1}}{q-q^{-1}}.$$
$$ E^r=0, \ \ \ F^r=0.$$
The Hopf algebra structure is defined by using the standard comultiplication formulas for $E$, $F$ and $K$, while $H$ is primitive, that is 
 $\Delta(H)=1 \otimes H+H \otimes 1$. Thus, the antipode map $S$ is induced by letting $S(H)=-H$ and $S(E)$, $S(K)$ and $S(F)$ are defined as usual. 
Not all $\U$  representations are of interest. We say that an $\U$-module $M$ is a {\em weight} module if $M$ is finite-dimensional and $H$-diagonalizable 
such that $q^K=H$ (as an operator on $M$). As in \cite{CGP}, we denote by $\mathcal{C}$ the category of weight $\U$-modules.
Irreducible objects in $\mathcal{C}$ are easy to classify. They are clearly of highest weight and belong to three types: $S_n$, $n=0,..,r-1$, of dimension 
$n+1$,  $V_{\alpha}$, where  $\alpha \in \ddot{\mathbb{C}}:=(\mathbb{C} \setminus \mathbb{Z}) \cup r \mathbb{Z}$ are of dimension $r$, and one-dimensional modules $\mathbb{C}^H_{k r}$. Then a complete list of irreps is given by: (atypicals) $S_i \otimes \mathbb{C}^H_{kr}$, $k \in \mathbb{Z}$,  $n=0,..,r-2$ and (typicals) $V_{\alpha}$, $\alpha \in (\mathbb{C} \setminus \mathbb{Z}) \cup r \mathbb{Z}$. All the irreducible modules can be constructed explicitly in terms of their bases \cite{CGP}.

\subsection{Beyond the category $\mathcal{C}$} Here we discuss an enlargement of the category $\mathcal{C}$. As we shall explain below, it is actually not true that the (full) category of finitely generated  $\mathcal{M}(r)$-modules is equivalent to $\mathcal{C}$.  This is why it is interesting and important to consider $\U$-modules outside the category $\mathcal{C}$.

Next we introduce the cateogry $\mathcal{C}_{log}$ (here $log$ is meant to indicate inclusion of {\em logarithmic} modules - this terminology is motivated by a related notion in logarithmic conformal field theory \cite{HLZ, Mil1}) . Objects in $\mathcal{C}_{log}$ are finite-dimensional $\mathcal{M}(p)$-modules such that $q^{H}=K$ (as operators) but $H$ does not act necessarily semisimple (of course, here $q^H=\sum_{n \geq 0} \frac{( \pi i H)^n}{r^n n!}$). 
Now we show that this category admits self-extensions of generic modules $V_{\alpha}$ but no self-extension of simple modules of dimension 
$< r$ (this is in agreement with the singlet algebra case \cite{AM1,Mil1}). Consider a $2r$-dimensional module $\tilde{V}_{\lambda}$ with a basis $v_{i}^0$, $v_{i}^0$, $i=0,...,r-1$
with the action:

\begin{align*}
& H.v_i^0=(\lambda-2i)v_i^0+v_i^1, \qquad H.v_i^1=(\lambda-2i)v_i^1; \qquad  0 \leq i \leq r-1  \\
& K.v_i^0=q^{\lambda-2i}v_i^0 +\frac{\pi i }{r} q^{\lambda-2i}v_i^1, \qquad Kv_i^1 =q^{\lambda -2i} v_i^1 \\
& E.v_i^0=\frac{\{ i \} \{\lambda+1-i \}}{\{i \}^2} v_{i-1}^{0}+\beta_i v_{i-1}^1, \qquad E.v_i^1=\frac{\{ i \} \{\lambda+1-i \}}{\{i \}^2} v_{i-1}^{1},\qquad
 E.v_0^0=E.v_0^1=0, \\
& F.v_i^0=v_{i+1}^0, \qquad F.v_i^1=v_{i+1}^1; \   0 \leq i \leq r-2, \qquad F.v_{r-1}^0=F.v_{r-1}^1=0,
\end{align*}
where 
$$\beta_0=0,\qquad \beta_i=\frac{\pi i}{r}\frac{1}{q-q^{-1}} \sum_{j \geq 1}^i \left( q^{\lambda-2(j-1)}+ q^{2(j-1)-\lambda}\right).$$ 
{\bf Claim:} $\tilde{V}_\lambda$ is a $\U$-module, a self-extension of ${V}_{\lambda}$.
In order to show that $\tilde{V}_\lambda$ is an $\U$-module the only non-trivial relation 
relation to verify is
$$(EF-FE)v_i^0=\frac{K-K^{-1}}{q-q^{-1}}v_i^0.$$
For this it is essential that 
$$\beta_r=\sum_{j=1}^r q^{2j}=0.$$
It is clear that we now get  a non-split short exact sequece:
$$0 \longrightarrow V_\lambda \longrightarrow \tilde{V}_\lambda \longrightarrow V_{\lambda} \longrightarrow 0.$$
Any irreducible modules of dimension $i<r$, is isomorphic to $S_i \otimes \mathbb{C}^H_{kr}$. In order to rule out  its self extension it is enough 
to follows steps in the proof of the claim and observe that $\sum_{j=0}^i q^{2j} \neq 0.$ 

We conclude with the conjecture that the category of finite-length $\M$-modules is equivalent to the category 
$\mathcal{C}_{log}$. { Moreover, the full subcategory $\mathcal{C} \subset \mathcal{C}_{log}$ is expected to be equivalent to 
the subcategory of $\mathcal{M}(r)$-mod generated by irreducible objects (as a tensor category). We also plan to investigate possible braided category structure on 
$\mathcal{C}_{log}$.}

\subsection{Projective modules}\label{sec:proj}

Projective modules in $\mathcal{C}$ are classified in \cite{CGP}. {Although this paper did not discuss projective covers, it can be easily shown that 
the projective modules denoted by $P_i$ are projective covers of $S_i$, $P_i \otimes \mathbb{C}^H_{kr}$ are projective covers of $S_i \otimes \mathbb{C}^H_{kr}$, 
and $V_\alpha$, $\alpha \in \ddot{\mathbb{C}}$, are their own projective covers.} Their 
Jordan-H\"older filtration is described by Figure 1.

\begin{figure}[h]
\begin{center}
\begin{tikzpicture}[scale=0.70][thick,>=latex,
nom/.style={circle,draw=black!20,fill=black!20,inner sep=1pt}
]
\node(left0) at  (-1,0) [] {$P_i \otimes \Ck: \ \ $}; 
\node (top1) at (5,2.5) [] {$ S_i \otimes \Ck $};
\node (left1) at (2.5,0) [] {$ S_{r-i-2} \otimes \mathbb C^H_{(k+1)r} $};
\node (right1) at (7.5,0) [] {$ S_{r-i-2} \otimes \mathbb C^H_{(k-1)r} $};
\node (bot1) at (5,-2.5) [] {$  S_i \otimes \Ck$};
\draw [->] (top1) -- (left1);
\draw [->] (top1) -- (right1);
\draw [->] (left1) -- (bot1);
\draw [->] (right1) -- (bot1);

\node (top2) at (12, 2) [] {$V_{1+i+(k-1)r}$};
\node (bot2) at (12, -2) [] {$V_{r-i-1+kr}$};
\draw [->] (top2) -- (bot2);

\end{tikzpicture}
\captionbox{\label{fig:Loewy} Loewy diagram of $P_i \otimes \Ck$ in terms of simple (left) and typical (right) composition factors}{\rule{12cm}{0cm}}
\end{center}
\end{figure}

We have ${\rm End}_{\U}(V_\alpha)=\mathbb{C} {\rm Id}_{V_{\alpha}}$ and it can be shown that ${\rm End}_{\U}(P_i)=\mathbb{C} {\rm Id}_{P_i\otimes \Ck} \oplus \mathbb{C}x_{i, k}$ where $x_{i, k}$ is the nilpotent endomorphism of $P_i\otimes \Ck$ uniquely determined by $w^H_i \mapsto w_i^S$ (see Figure 1).

\subsection{Modified quantum dimension}
Let $\Proj$ be the full subtensor category of modules generated by projective $\U$-modules. There exists a unique trace on $\Proj$, up to multiplication by an element of $\mathbb{C}$. In particular, there is a unique trace $t = \{t_V \}$, on $\Proj$, such that for any
$f \in {\rm End}_{\mathcal{C}} (V_0)$ we have $t_{V_0}(f) = (-1)^{r-1} \langle f \rangle$. For such choice of $t$, following \cite{CGP}, we define the modified quantum dimension function as ${\bf d} : {\rm Ob}(\Proj) \rightarrow \mathbb{C}$, ${\bf d}(V): =t_V({\rm Id}_V)$.

\section{Logarithmic invariants via deformable modules}

Open Hopf link invariants have been computed in \cite{CGP}. However for those involving projective modules knowledge of the tensor ring is required. 
Here we find a new way of computation that does not require this knowledge. For this, we introduce a deformable family of modules that then will be used to compute logarithmic tangle invariants.

\begin{theorem}\label{thm:deform}
Let $\epsilon \in (-\frac{1}{2},\frac{1}{2})$, $i \in \{0,...,r-2\}$ and let $\ell \in \Z$. Denote by $X_{\epsilon}$ the module with vector space basis $\{\w^L_{i+2-2r},\w^L_{i+4-2r},...,\w^L_{-i-2},\w^H_{-i},...,\w_{i}^H,\w_{-i}^S,...,\w_{i}^S,\w_{i+2}^R,...,\w^R_{2r-2-i}\}$ and action given by

\begin{align*}
\w_{i+2}^R=(-1)^{\ell}E\w_i^H, \qquad \w_{-i-2}^L=F\w_{-i}^H, \qquad F\w_{i+2}^R=\w_i^S+[1+i][\epsilon]\w_i^H
\end{align*}
\begin{align*}
\w_{i-2k}^H&=F^k\w_i^H & &\text{and} &\quad \w_{i-2k}^S&=F^k\w_i^S &  &\mathrm{for}\ k \in \{0,...,i\}\\
\w_{-i-2-2k}^L&=F^k\w_{-i-2}^L & &\text{and} &\quad \w_{i+2+2k}^R&=(-1)^{k \ell}E^k\w_{i+2}^R&  &\mathrm{for} \ k \in \{0,...,r-2-i\}
\end{align*}
\begin{align*}
H\w_k^X&=(k+\ell r+\epsilon)\w_k^X, & \quad K\w_k^X&=(-1)^{\ell}q^{k+\epsilon}\w_k^X,\qquad &\mathrm{for}\  X \in \{L,H,S,R\}\\
E\w_{k}^R&=(-1)^{\ell}\w_{k+2}^R, & \quad F\w_k^X&=\w_{k-2}^X, \qquad &\mathrm{for}\ X \in \{L,S,H\} \\
F\w_{-i}^S&=[1+i][\epsilon]\w_{-i-2}^L,& E\w_{i}^S&=2(-1)^{\ell+1}[1+i][\epsilon]\w_{i+2}^R,\qquad  & E\w_{2r-2-i}^R=F\w_{i+2-2r}^L=0
\end{align*}
\begin{align*}
E\w_{-i-2}^L&=2(-1)^{\ell}[i+1][\epsilon]\w_{-i}^H+(-1)^{\ell}\w_{-i}^S,\\
E\w_{-i-2-2k}^L&=(-1)^{\ell+1}[1+i+k][k-\epsilon]\w_{-i-2+2(k-1)}^L,\\
F\w_{i+2+2k}^R&=-[1+i+k][k+\epsilon]\w_{i+2+2(k-1)}^R,
\end{align*}
\begin{align*}
E\w_{i-2k}^H&=(2[1+i-k+\epsilon][k]-[1+i-k][k-\epsilon ])(-1)^{\ell}\w_{i-2(k-1)}^H+(-1)^{\ell}\w_{i-2(k-1)}^S,\\
E\w_{i-2k}^S&=(2[1+i-k][k-\epsilon]-[1+i-k+\epsilon][k])(-1)^{\ell}\w_{i-2(k-1)}^S+2(-1)^{\ell+1}[1+i]^2[\epsilon]^2\w_{i-2(k-1)}^H.
\end{align*}
Then
\[
X_{\epsilon}= \left\{ \begin{array}{cc}
V_{1+i-r+\ell r+\epsilon} \oplus V_{-1-i+r+\ell r+\epsilon} & \text{ if }\epsilon \neq 0\\
P_i  \otimes \CC_{\ell r}^H& \text{ if } \epsilon = 0\\
\end{array} \right.\]
\end{theorem}

\begin{proof}
Let $\epsilon \in (-\frac{1}{2},\frac{1}{2})\setminus \{0\}$. Let $\{x_0,...,x_{r-1}\}$ denote the standard basis for $ V_{1+i-r+\ell r+\epsilon}$ and $\{y_0,...,y_{r-1}\}$ the standard basis for $V_{-1-i+r+\ell r+\epsilon}$. Define a new basis $\{\w^L_{i+2-2r},\w^L_{i+4-2r},... ,\w^L_{-i-2},$ $\w^H_{-i},...,\w_{i}^H,\w_{-i}^S,...,\w_{i}^S,\w_{i+2}^R,...,\w^R_{2r-2-i}\}$ for $V_{1+i-r+\ell r+\epsilon} \oplus V_{-1-i+r+\ell r+\epsilon}$ by 
\begin{align*}
 \w_{i-2k}^H&=2x_k -\frac{1}{2[1+i][\epsilon]}y_{r-1-i+k}  \qquad \w_{i-2k}^S=-2[1+i][\epsilon]x_k+y_{r-1-i+k}\\
\w_{-i-2-2k}^L&=2x_{1+i+k} \qquad \w_{i+2+2k}^R=\frac{-1}{2[1+i][\epsilon]}\left( \prod\limits_{s=0}^{k} [1+i+s][-s-\epsilon] \right)y_{r-2-i-k}
\end{align*}

We will first prove the statement for $\epsilon \not = 0$. We will show that the action of $\U$ on the new basis is precisely as stated in the theorem. Recall that the action on the standard basis element $v_k \in V_{\alpha}$ is given by 
\[ Hv_k=(\alpha+r-1-2k)v_k, \qquad Ev_{k}=[k][k-\alpha]v_{k-1}, \qquad Fv_k=v_{k+1}. \]
By direct computation, we obtain the following:
\begin{equation*}
\begin{split}
E\w_i^H&=2Ex_0-\frac{1}{2[1+i][\epsilon]}Ey_{r-1-i}=-\frac{[r-1-i][- \ell r-\epsilon]}{2[1+i][\epsilon]}y_{r-2-i}\\
&=\frac{1}{2}(-1)^{\ell}y_{r-2-i}=(-1)^{\ell}\w_{i+2}^R,\\
F\w_{-i}^H&=2Fx_i-\frac{1}{2[1+i][\epsilon]}Fy_{r-1}=2x_{i+1}=\w_{-i-2}^L,\\
\w_i^S+[1+i][\epsilon]\w_i^H&=-2[1+i][\epsilon]x_0+y_{r-1-i}+2[1+i][\epsilon]x_0 -\frac{1}{2}y_{r-1-i}=\frac{1}{2}y_{r-1-i}=\frac{1}{2}Fy_{r-2-1}\\
&=F\w_{i+2}^R.
\end{split}
\end{equation*}
Hence, we have shown $\w_{i+2}^R=(-1)^{\ell}E\w_i^H$, $\w_{-i-2}^L=F\w_{-i}^H$, and $F\w_{i+2}^R=\w_i^S + [1+i][\epsilon]\w_i^H$. It is easily seen that $F\w_k^X=\w_{k-2}^X$ for all $X \in \{L,S,H\}$ and that
\begin{align*}
E\w_{i+2+2k}^R&=\frac{-1}{2[1+i][\epsilon]}\left( \prod\limits_{s=0}^{k} [1+i+s][-s-\epsilon] \right)Ey_{r-2-i-k}\\
&=\frac{-[r-2-i-k][-1-k-\ell r-\epsilon]}{2[1+i][\epsilon]}\left( \prod\limits_{s=0}^{k} [1+i+s][-s-\epsilon] \right)y_{r-2-i-(k+1)}\\
&=\frac{-(-1)^{\ell}}{2[1+i][\epsilon]}\left( \prod\limits_{s=0}^{k+1} [1+i+s][-s-\epsilon] \right)y_{r-2-i-(k+1)}=(-1)^{\ell}\w_{i+2+2(k+1)}^R
\end{align*}
so $E\w_k^R=(-1)^{\ell}\w_{k+2}^R$, which gives $\w_{i+2+2k}^R=(-1)^{k\ell}E^k\w_{i+2}^R$, $\w_{i-2k}^H=F^k\w_i^H$, $\w_{i-2k}^S=F^k\w_i^S$, and $\w_{-i-2-2k}^L=F^k\w_{-i-2}^L$.\\
$H$ acts on a standard basis vector $v_k \in V_{\alpha}$ by $Hv_k=(\alpha+r-1-2k)v_k$, so we have
\begin{align*}
Hx_k&=(1+i-r+\ell r+\epsilon +r -1-2k)x_k=(i-2k+\ell r+\epsilon)x_k\\
Hy_{r-1-i+k}&=(-1-i+r+\ell r+\epsilon +r-1-2(r-1-i+k))y_{r-1-i+k}\\ &=(i-2k+\ell r+\epsilon)y_{r-1-i+k}\\
Hx_{1+i+k}&=(1+i-r+\ell r+\epsilon +r-1-2(1+i+k))x_{1+i+k}\\&=(-i-2-2k+\ell r+\epsilon)x_{1+i+k}\\
Hy_{r-1-i-(k+1)}&=(-1-i+r+\ell r+\epsilon +r-1-2(r-1-i-(k+1))y_{r-1-i-(k+1)}\\
&=(i+2+2k+\ell r+\epsilon)y_{r-1-i-(k+1)}\\
\end{align*}
From this, it immediately follows that
\begin{align*}
H\w_{i-2k}^H&=(i-2k+\ell r+\epsilon)\w_{i-2k}^H,\qquad\qquad\qquad\qquad
H\w_{i-2k}^S=(i-2k+\ell r+\epsilon)\w_{i-2k}^S,\\
H\w_{-i-2-2k}^L&=(-i-2-2k + \ell r+\epsilon)\w_{-i-2-2k}^L,\qquad\quad
H \w_{i+2+2k}^R=(i+2+2k+\ell r+\epsilon) \w_{i+2+2k}^R
\end{align*}
and $K$ acts as $q^H$, so we have shown that $H \w_k^X=(k+\ell r+\epsilon)\w_k^X$ and $K \w_k^X=q^{k+\ell r+\epsilon}\w_k^X=(-1)^{\ell}q^{k+\epsilon}\w_k^X$ for all $X \in \{L,H,S,R\}$. It is easy to see that $E\w_{2r-2-i}^R=F\w_{i+2-2r}^L=0$ as $Fx_{r-1}=Ey_0=0$. We also have
\begin{align*}
F\w_{-i}^S&=-2[1+i][\epsilon]Fx_i+Fy_{r-1}=-2[1+i][\epsilon]x_{i+1}=-[1+i][\epsilon]\w_{-i-2}^L,\\
E\w_i^S&=-2[1+i][\epsilon]Ex_0+Ey_{r-1-i}=[1+i][-\ell r-\epsilon]y_{r-2-i}=2(-1)^{\ell+1}[1+i][\epsilon]\w_{i+2}^R,\\
E\w_{-i-2-2k}^L&=2Ex_{1+i+k}=2[1+i+k][r+k-\ell r-\epsilon]x_{i+k}\\&=(-1)^{\ell+1}[1+i+k][k-\epsilon]\w_{-i-2-2(k-1)}^L,\\
F\w_{i+2+2k}^R&=\frac{-1}{2[1+i][\epsilon]}\left( \prod\limits_{s=0}^{k} [1+i+s][-s-\epsilon] \right)y_{r-1-i-k}\\
&=\frac{-[1+i+k][-k-\epsilon]}{2[1+i][\epsilon]}\left( \prod\limits_{s=0}^{k-1} [1+i+s][-s-\epsilon] \right)y_{r-2-i-(k-1)}\\
&=-[1+i+k][k+\epsilon]\w_{i+2+2(k-1)}^R
\end{align*}
and
\begin{align*}
(-1)^{\ell}\left( \w_{-i}^S+2[1+i][\epsilon]\w_{-i}^H \right) &=(-1)^{\ell} \left(-2[1+i][\epsilon]x_i+y_{r-1}+4[1+i][\epsilon]x_i-y_{r-1}\right)\\
&=2(-1)^{\ell}[1+i][\epsilon]x_i=2[1+i][\ell r+\epsilon]x_i=E\w_{-i-2}^L.
\end{align*}
From the definition of $\w_{i-2k}^H$ and $\w_{i-2k}^S$ it is easy to show that
\begin{align*}
x_k&=\w_{i-2k}^H+\frac{1}{2[1+i][\epsilon]}\w_{i-2k}^S,\qquad\qquad
y_{r-1-i+k}&=2\w_{i-2k}^S+2[1+i][\epsilon]\w_{i-2k}^H.
\end{align*}
From this, we see that
\begin{align*}
E\w_{i-2k}^H&=2Ex_k-\frac{1}{2[1+i][\epsilon]}Ey_{r-1-i+k}\\&=2[k][1+i-k+\ell r+\epsilon]x_{k-1}-\frac{[1+i-k][k-\ell r-\epsilon]}{2[1+i][\epsilon]}y_{r-2-i+k}\\
&=2[k][1+i-k+\epsilon](-1)^{\ell}\left(\w_{i-2(k-1)}^H+\frac{1}{2[1+i][\epsilon]}\w_{i-2(k-1)}^S \right)\\
&\qquad -\frac{[1+i-k][k-\epsilon]}{2[1+i][\epsilon]}(-1)^{\ell}\left(2\w_{i-2(k-1)}^S+2[1+i][\epsilon]\w_{i-2(k-1)}^H\right)\\
&=\left( 2[k][1+i-k+\epsilon]-[1+i-k][k-\epsilon] \right)(-1)^{\ell}\w_{i-2(k-1)}^H\\
&\qquad+\left(\frac{[k][1+i-k+\epsilon]-[1+i-k][k-\epsilon]}{[1+i][\epsilon]} \right)(-1)^{\ell} \w_{i-2(k-1)}^S
\end{align*}
and
\begin{align*}
E\w_{i-2k}^S&=-2[1+i][\epsilon]Ex_k+Ey_{r-1-i+k}\\
&=-2[1+i][\epsilon][k][1+i-k+\ell r+\epsilon]x_{k-1}+[1+i-k][k-\ell r-\epsilon]y_{r-2-i+k}\\
&=-2[1+i][\epsilon][k][1+i-k+\epsilon](-1)^{\ell}\left(\w_{i-2(k-1)}^H+\frac{1}{2[1+i][\epsilon]}\w_{i-2(k-1)}^S \right)\\
&\qquad+[1+i-k][k-\epsilon] (-1)^{\ell} \left( 2\w_{i-2(k-1)}^S+2[1+i][\epsilon]\w_{i-2(k-1)}^H\right)\\
&=\left( 2[1+i-k][k-\epsilon]-[1+i-k+\epsilon][k] \right)(-1)^{\ell} \w_{i-2(k-1)}^S\\
&\qquad+2[1+i][\epsilon] \left( [1+i-k][k-\epsilon]-[1+i-k+\epsilon][k] \right) (-1)^{\ell} \w_{i-2(k-1)}^H.
\end{align*}
However, by expanding the brackets one has
$[k][1+i-k+\epsilon]-[1+i-k][k-\epsilon]=[1+i][\epsilon]$.
Hence, the above equations give
\begin{align*}
\w_{i-2k}^H&=\left( 2[k][1+i-k+\epsilon]-[1+i-k][k-\epsilon] \right)(-1)^{\ell}\w_{i-2(k-1)}^H + (-1)^{\ell}\w_{i-2(k-1)}^S,\\
\w_{i-2k}^S&=\left( 2[1+i-k][k-\epsilon]-[1+i-k+\epsilon][k] \right) (-1)^{\ell}\w_{i-2(k-1)}^S+2(-1)^{\ell+1}[1+i]^2[\epsilon]^2 \w_{i-2(k-1)}^H 
\end{align*}
as desired. This proves that $X_{\epsilon}=V_{1+i-r+\ell r+\epsilon} \oplus V_{-1-i+r+\ell r+\epsilon}$ when $\epsilon \not = 0$. As $\epsilon \to 0$, it is easy to see that the action on $X_0$ is exactly the action on $P_i \otimes \mathbb{C}_{\ell r}^H$ (see \cite{CGP}) by identifying $\w_k^X \in X_0$ with $\w_k^X \otimes v \in P_i \otimes \mathbb{C}_{kr}^H$ (here $\mathbb{C}_{kr}^H=\mathbb C v$). Hence, we have shown
\[
X_{\epsilon}= \left\{ \begin{array}{cc}
V_{1+i-r+\ell r+\epsilon} \oplus V_{-1-i+r+\ell r+\epsilon} & \text{ if }\epsilon \neq 0\\
P_i  \otimes \mathbb{C}_{\ell r}^H& \text{ if } \epsilon = 0\\
\end{array} \right. .
\]
\end{proof}

\begin{corollary}
$\lim\limits_{\epsilon \to a} AX_{\epsilon}=A \lim\limits_{\epsilon \to a} X_{\epsilon} \quad \forall A \in \overline{U}^H_q(\mathfrak{sl}_2)$
\end{corollary}

\begin{proof}
The Theorem holds for $E, F, H$ and $K$ by construction of $X_{\epsilon}$ and hence holds for all polynomials in $E,F,H$ and $K$. 
\end{proof}

\begin{corollary}
Morphisms that consist of compositions of braidings, twists, evaluations and co-evaluations commute with limits.
\end{corollary}

\begin{proof}
This follows from the previous Theorem as braiding, twist, evaluation and co-evaluation are expressed in terms of the action of elements of $\U$ or 
in the $H$-completion  of $\U$ (for the braiding). 
\end{proof}

\begin{proposition}
The modified quantum dimension satisfies
$\lim\limits_{\epsilon \rightarrow 0} \textbf{d}(X_{\epsilon})=\textbf{d}(X_0)$.
\end{proposition}

\begin{proof}
Set $\lambda= 1+i-r$, then we have
\begin{align*}
 \textbf{d}(X_{\epsilon})& = \textbf{d}(V_{\lambda+\ell r+\epsilon} \oplus V_{-\lambda+\ell r+\epsilon})= \textbf{d}(V_{\lambda+\ell r+\epsilon})+\textbf{d}(V_{-\lambda + \ell r+\epsilon})) \\
 & =(-1)^{r-1}r \left( \frac{\{\lambda+\ell r+\epsilon\}}{\{r(\lambda+\ell r + \epsilon)\}}+ \frac{\{-\lambda+\ell r+ \epsilon\}}{\{r(-\lambda+\ell r+\epsilon)\}} \right)\\
  & =(-1)^{\ell(r-1)+r-1}r \left( \frac{\{\lambda+\epsilon\}}{\{r(\lambda+ \epsilon)\}}+ \frac{\{-\lambda+ \epsilon\}}{\{r(-\lambda+\epsilon)\}} \right)\\
 &\hspace{-1.3cm}=(-1)^{(\ell+1)(r-1)}r \frac{(q^{\lambda+\epsilon}-q^{-(\lambda+\epsilon)})(q^{r(-\lambda+\epsilon)}-q^{-r(-\lambda+\epsilon)})+(q^{-\lambda+\epsilon}-q^{-(-\lambda+\epsilon)})(q^{r(\lambda+\epsilon)}-q^{-r(\lambda+\epsilon)})}{(q^{r(\lambda+\epsilon)}-q^{-r(\lambda+\epsilon)})(q^{r(-\lambda+\epsilon)}-q^{-r(-\lambda+\epsilon)})}. 
\end{align*}
We have to evaluate this expression for $\epsilon \rightarrow 0$. Both the denominator and nominator vanish in this limit. It turns out that the same happens for the derivatives of both denominator and nominator and so we have to apply the rule of L'H\^opital twice and we get as nominator $8r(-1)^{1+i-r}(q^{1+i-r}+q^{-(1+i-r)})$, and a denominator of $8r^2$. Hence,
\begin{align*}
 \lim\limits_{\epsilon \rightarrow 0} \textbf{d}(X_{\epsilon})&=(-1)^{\ell(r-1)+r-1}r \left( \frac{8r(-1)^{1+i-r}(q^{1+i-r}+q^{-(1+i-r)})}{8r^2} \right) \\
 &=(-1)^{\ell(r-1)+i+1}(q^{i+1}+q^{-i-1})=\textbf{d}(P_i \otimes \mathbb{C}_{\ell r}^H).
\end{align*}
\end{proof}

\subsection{Colored Alexander invariants}

We now apply the above construction. For this we need to recall Section 3 of \cite{GPT}. Let $L$ be a $\mathcal C$-colored ribbon graph, such that at least one of the colors is a simple $V_\lambda$. 
Let $T_\lambda$ be the colored $(1, 1)$-ribbon graph by cutting the edge belonging to $V_\lambda$. Then the re-normalized Reshetikhin-Turaev link invariant is
\[
F'(L) = t_{V_{\lambda}}(T_\lambda).
\]
These where shown in \cite{GPT} to coincide with Murakami's Alexander invariants \cite{M1} provided all colors are simple projective modules. We can now extend these results to any projective module. For this let $L$ be as above but with at least one of the colors the module $P_i\otimes  \CC_{\ell r}^H$. Let $T_{P_i\otimes  \CC_{\ell r}^H}$ be the colored $(1, 1)$-ribbon graph by cutting the edge belonging to $P_i\otimes  \CC_{\ell r}^H$. Let $T_\lambda$ be the colored $(1, 1)$-ribbon graph obtained by replacing the open strand (which is colored with $P_i\otimes  \CC_{\ell r}^H$) with $V_\lambda$. 

\begin{theorem}\label{thm:links}
The colored $(1, 1)$-ribbon graph $T_{P_i\otimes  \CC_{\ell r}^H}$ satisfies
\[
 t_{P_i\otimes  \CC_{\ell r}^H}\left(T_{P_i\otimes  \CC_{\ell r}^H}\right) = \lim\limits_{\epsilon \to 0}\left( t_{V_{1+i-r+\ell r+\epsilon}}\left(T_{1+i-r+\ell r+\epsilon}\right) +   t_{V_{-1-i-r+\ell r+\epsilon}}\left(T_{-1-i-r+\ell r+\epsilon}\right) \right)
\]
and 
\[
T_{P_i\otimes  \CC_{\ell r}^H} =  a Id_{P_i \otimes  \CC_{\ell r}^H} +bx_{i, \ell r} 
\]
with coefficients
\[
a =  \lim\limits_{\epsilon \to 0} \frac{t_{V_{-1-i-r+\ell r+\epsilon}}\left(T_{-1-i-r+\ell r+\epsilon}\right)}{\textbf{d}\left(V_{-1-i+r+\ell r+\epsilon}\right)} =  \lim\limits_{\epsilon \to 0}  \frac{t_{V_{1+i-r+\ell r+\epsilon}}\left(T_{1+i-r+\ell r+\epsilon}\right)}{\textbf{d}\left(V_{1+i-r+\ell r+\epsilon} \right)}
\]
and 
\begin{equation}\nonumber
\begin{split}
b &= \lim_{\epsilon \to 0}\frac{-1}{[1+i][\epsilon]}\left(  \frac{t_{V_{-1-i-r+\ell r+\epsilon}}\left(T_{-1-i-r+\ell r+\epsilon}\right)}{\textbf{d}\left(V_{-1-i+r+\ell r+\epsilon}\right)} -   \frac{t_{V_{1+i-r+\ell r+\epsilon}}\left(T_{1+i-r+\ell r+\epsilon}\right)}{\textbf{d}\left(V_{1+i-r+\ell r+\epsilon} \right)} \right)\\ 
&= \frac{r}{2\pi i \{1+i\}} \left(\frac{d}{d\lambda} \frac{t_{V_{\lambda}}\left(T_{\lambda}\right)}{\textbf{d}\left(V_{\lambda}\right)} \Big|_{\lambda=\ell r+r-i-1} 
 -\frac{d}{d\lambda} \frac{t_{V_{\lambda}}\left(T_{\lambda}\right)}{\textbf{d}\left(V_{\lambda} \right)}\Big|_{\lambda=1+i-r+\ell r} \right). 
\end{split}
\end{equation}
\end{theorem}
Before proving this theorem, we remark that this result nicely relates to the work of Murakami and Nagatomo on logarithmic link invariants obtained using different quantum groups but the same $R$-matrix \cite{M1, M2, MN}.

We also remark that
\[
T_\lambda = \frac{t_{V_{\lambda}}\left(T_{\lambda}\right)}{\textbf{d}\left(V_{\lambda}\right)} \, Id_{V_\lambda}.
\]
\begin{proof}

For the first statement, we use the identity $\lim\limits_{\epsilon \rightarrow 0} V_{1+i-r+\ell r+\epsilon} \oplus V_{-1-i+r+ \ell r +\epsilon}=P_i \otimes \CC_{\ell r}^H$ which gives
\begin{align*}
t_{P_i\otimes  \CC_{\ell r}^H}\left(T_{P_i\otimes  \CC_{\ell r}^H}\right) = \lim\limits_{\epsilon \to 0} T_{V_{1+i-r+\ell r+\epsilon} \oplus V_{-1-i+r+ \ell r +\epsilon}}\left(T_{V_{1+i-r+\ell r+\epsilon} \oplus V_{-1-i+r+ \ell r +\epsilon}} \right)
\end{align*}
where we pulled the limit out of the function using the fact that limits commute with the partial trace and the action of $\U$. The relation 
follows since coloring with the direct sum of two objects $X\oplus Y$ amounts to computing the sum of the individually colored components
\begin{equation*}
\begin{split}
t_{X \oplus Y}&\left(T_{X \oplus Y} \right)=  t_{X}\left(T_{X} \right) + t_{Y}\left(T_{Y} \right). 
\end{split}
\end{equation*} 

For the second statement, since $\textsf{w}_i^H$ generates $P_i\otimes \CC_{\ell r}^H$, it is enough to find the action of $T_{P_i\otimes  \CC_{\ell r}^H}$ on $\textsf{w}_i^H$. 
For this we compute the action of $T_{V_{1+i-r+\ell r+\epsilon} \oplus V_{-1-i+r+ \ell r +\epsilon}}$ on $\textsf{w}_i^H$ and then take the limit $\epsilon$ to zero. 
Recall that we have $\textsf{w}_i^H=a_{\epsilon}x_0+b_{\epsilon}y_{r-1-i}$ and $\textsf{w}_i^S=c_{\epsilon}x_0+d_{\epsilon}y_{r-1-i}$
where $x_0$, $y_{r-1-i}$, $a_{\epsilon}= 2$, $b_{\epsilon}=-\frac{1}{2[1+i][\epsilon]}$,  $c_{\epsilon}=-2[1+i][\epsilon]$, and $d_{\epsilon}=1$ are as in the construction of $X_{\epsilon}$. Notice that $x_0=d_{\epsilon}\w_i^H-b_{\epsilon}\w_i^S$ and $y_{r-1-i}=a_{\epsilon}\w_i^S-c_{\epsilon}\w_i^H$. We can now compute the action of $T_{P_i\otimes  \CC_{\ell r}^H}$ on $\textsf{w}_i^H$: 
\begin{equation*}
\begin{split}
T_{P_i\otimes  \CC_{\ell r}^H}(\w_i^H ) &= T_{P_i\otimes  \CC_{\ell r}^H}   \left(\lim\limits_{\epsilon \to 0} a_{\epsilon}x_0+b_{\epsilon}y_{r-1-i}\right)
= \lim\limits_{\epsilon \to 0} \left( a_\epsilon T_{1+i-r + \ell r+\epsilon}(x_0) +  b_\epsilon T_{-1-i+r + \ell r+\epsilon}(y_{r-1-i})\right)  \\
&=  \lim\limits_{\epsilon \to 0} \left( a_\epsilon\frac{t_{V_{1+i-r+\ell r+\epsilon}}\left(T_{1+i-r+\ell r+\epsilon}\right)}{\textbf{d}\left(V_{1+i-r+\ell r+\epsilon} \right)}(d_{\epsilon}\w_i^H-b_{\epsilon}\w_i^S) + \right.\\
&\qquad\qquad \left. b_\epsilon  \frac{t_{V_{-1-i-r+\ell r+\epsilon}}\left(T_{-1-i-r+\ell r+\epsilon}\right)}{\textbf{d}\left(V_{-1-i+r+\ell r+\epsilon}\right)}(a_{\epsilon}\w_i^S-c_{\epsilon}\w_i^H)\right).
\end{split}
\end{equation*} 
It follows that 
\begin{equation*}
\begin{split}
a &=  \lim\limits_{\epsilon \to 0} \left( a_\epsilon d_\epsilon \frac{t_{V_{1+i-r+\ell r+\epsilon}}\left(T_{1+i-r+\ell r+\epsilon}\right)}{\textbf{d}\left(V_{1+i-r+\ell r+\epsilon} \right)}- b_{\epsilon}c_{\epsilon}\frac{t_{V_{-1-i-r+\ell r+\epsilon}}\left(T_{-1-i-r+\ell r+\epsilon}\right)}{\textbf{d}\left(V_{-1-i+r+\ell r+\epsilon}\right)} \right)\\
&=\lim\limits_{\epsilon \to 0} \left( 2 \frac{t_{V_{1+i-r+\ell r+\epsilon}}\left(T_{1+i-r+\ell r+\epsilon}\right)}{\textbf{d}\left(V_{1+i-r+\ell r+\epsilon} \right)} - \frac{t_{V_{-1-i-r+\ell r+\epsilon}}\left(T_{-1-i-r+\ell r+\epsilon}\right)}{\textbf{d}\left(V_{-1-i+r+\ell r+\epsilon}\right)} \right)
\end{split}
\end{equation*} 
and 
\begin{equation*}
\begin{split}
b &= - \lim\limits_{\epsilon \to 0} \left(a_\epsilon b_\epsilon \frac{t_{V_{1+i-r+\ell r+\epsilon}}\left(T_{1+i-r+\ell r+\epsilon}\right)}{\textbf{d}\left(V_{1+i-r+\ell r+\epsilon} \right)} - a_{\epsilon}b_{\epsilon}\frac{t_{V_{-1-i-r+\ell r+\epsilon}}\left(T_{-1-i-r+\ell r+\epsilon}\right)}{\textbf{d}\left(V_{-1-i+r+\ell r+\epsilon}\right)} \right)\\
&=  \lim\limits_{\epsilon \to 0} \frac{1}{[1+i][\epsilon]} \left( \frac{t_{V_{1+i-r+\ell r+\epsilon}}\left(T_{1+i-r+\ell r+\epsilon}\right)}{\textbf{d}\left(V_{1+i-r+\ell r+\epsilon} \right)} - \frac{t_{V_{-1-i-r+\ell r+\epsilon}}\left(T_{-1-i-r+\ell r+\epsilon}\right)}{\textbf{d}\left(V_{-1-i+r+\ell r+\epsilon}\right)} \right).
\end{split}
\end{equation*} 
Evaluating the limits (using L'H\^opital's rule for $b$) give the result.  
\end{proof}

Recall the definition of general Hopf link invariants. Given two modules $V,W$ in $\mathcal{C}$, define 
$$\Phi_{V,W}=({\rm Id}_W \otimes ev'_V) \circ (c_{V,W} \otimes {\rm Id}_{V^*}) \circ (c_{W,V} \otimes {\rm Id}_{V^*}) \circ ({\rm Id}_W \otimes coev_V) \in {\rm End}(W).$$
where $ev'_V$ and $coev_V$ are the right evaluation and left coevaluation, respectively.
These invariants have been computed in \cite{CGP}. If $W$ is simple then this computation is relatively straight forward as one only needs to know the action on a highest-weight state. One gets for example
\begin{equation}\label{eq:hopflinks}
\begin{split}
\Phi_{V_\alpha, V_\beta} &= \frac{\{r\beta\}}{\{\beta\}} q^{\alpha\beta} {\rm Id}_{V_\beta}, \qquad\qquad \Phi_{S_i\otimes \Ck, V_\beta} =  \frac{\{(i+1)\beta\}}{\{\beta\}} q^{kr\beta} {\rm Id}_{V_\beta},\\
\Phi_{P_i\otimes \Ck, V_\beta} &=  \frac{\{ r\beta\}}{\{\beta\}} q^{kr\beta} \left( q^{(r-1-i)\beta}+q^{-(r-1-i)\beta}\right) {\rm Id}_{V_\beta}.
\end{split}
\end{equation}
Also note that $\textbf{d}\left(V_\beta\right) = \frac{\{ \beta \} }{\{ r \beta \}} (-1)^{r-1}r$.
If $W$ is not simple the computation is more involved and \cite{CGP} needs tensor product decomposition. Using above theorem and \eqref{eq:hopflinks} we immediately get
\begin{corollary}\label{cor:logHopf}
The colored Hopf links $\Phi_{Z,P_j}$ with $Z \in \{S_i \otimes \Ck, V_{\alpha}, P_i \otimes \Ck  \}$ are given by $$\Phi_{Z,P_j\otimes \mathbb C^H_{\ell r}}=a_Z{\rm Id}_{P_j\otimes \mathbb C^H_{\ell r}} + b_Z x_{j, \ell},$$ where $a_{P_i\otimes \Ck}=a_{V_{\alpha}}=0$ and 
\begin{equation*}
\begin{split}
a_{S_i\otimes \Ck}&=(-1)^{i+\ell i +k\ell r} \frac{\{(i+1)(j+1)\}}{\{j+1\}},\\
b_{S_i\otimes \Ck} &=(-1)^{i+\ell i +k\ell r}\frac{i\{(i+2)(j+1)\}-(i+2)\{i(j+1)\}}{[j+1]^2\{j+1\}},\\
b_{V_{\alpha}}&=q^{\alpha\ell r}\frac{(-1)^{r-j}r}{[j+1]^2}(q^{(r-1-j)\alpha}+q^{-(r-1-j)\alpha}),\\
b_{P_i\otimes \Ck}&=\frac{2r(-1)^{i+\ell i +k\ell r}}{[j+1]^2}(q^{(i+1)(j+1)}+q^{-(i+1)(j+1)}).
\end{split}
\end{equation*}
\end{corollary}

\section{Open Hopf links and asymptotic dimensions}\label{sec:comparison}

We will see that continuous quantum/asymptotic dimensions of the singlet algebra correspond to semi-simple parts of logarithmic Hopf links of $\U$ and the discrete ones correspond to the nilpotent ones; thus giving a categorical interpretation of asymptotic dimensions as suggested by \cite{CM, CMW} (see also \cite{BFM} for another derivation of asymptotic dimensions).

\begin{proposition}\label{prop:comp}
The quantum dimensions of the typical and atypical modules for the singlet vertex algebra are in agreement with the modified traces of their corresponding $\overline{U}_q^H(sl_2)$ modules in the following sense:
Let $\beta \in \ddot{\mathbb{C}}$ and $i, j \in \{0, \dots, r-2\}$ and $k, k' \in \mathbb Z$ then for $Re(\epsilon)>B_\epsilon^r$ and $\alpha=\sqrt{2r}i\epsilon$
\[
\text{\rm qdim}[F^{\frac{-i \alpha}{\sqrt{2r}}}_{\frac{\beta + r -1}{\sqrt{2r}}}]=\frac{t_{V_{\alpha}}(\Phi_{V_{\beta},V_{\alpha}})}{t_{V_{\alpha}}(\Phi_{S_0,V_{\alpha}})}, \qquad \text{\rm qdim}[M_{1-k,j+1}^{\frac{-i \alpha}{\sqrt{2r}}}]=\frac{t_{V_{\alpha}}(\Phi_{S_j \otimes \Ck, V_{\alpha}})}{t_{V_{\alpha}}(\Phi_{S_0 \otimes V_{\alpha}})},
\]
and for $Re(\epsilon)<B_\epsilon^r$ and $\epsilon \in \mathbb{S}(k,j+1+r(k+1))$ then
\[
  \text{\rm qdim}[F^{\epsilon}_{\frac{\beta + r -1}{\sqrt{2r}}}]=\frac{t_{P_j \otimes \Ck}(\Phi_{V_\beta, P_j \otimes \Ck} \circ x_{j,k})}{t_{P_j \otimes \Ck}(\Phi_{S_0,P_j \otimes \Ck} \circ x_{j,k})},
\qquad 
\text{\rm qdim}[M^{\epsilon}_{1-k',i+1}]=\frac{t_{P_j \otimes \Ck}(\Phi_{S_i\otimes \mathbb{C}_{k'r}^H,P_j \otimes \Ck} \circ x_{j,k})}{t_{P_j \otimes \Ck}(\Phi_{S_0,P_j \otimes \Ck} \circ x_{j,k})}.
\]
\end{proposition}

\begin{proof}
The proof follows directly by comparing Corollary \ref{cor:logHopf} and the quantum dimensions listed in section \ref{sec:singlet}, namely
\begin{align*}
\text{\rm qdim}[F^{\frac{-i \alpha}{\sqrt{2r}}}_{\frac{\beta + r -1}{\sqrt{2r}}}]&=e^{\pi \frac{-i \alpha}{\sqrt{2r}}(2\frac{\beta + r -1}{\sqrt{2r}}-\alpha_0)}\frac{(e^{- \pi \sqrt{2r}\frac{-i \alpha}{\sqrt{2r}}i^2}-e^{\pi \sqrt{2r}\frac{-i \alpha}{\sqrt{2r}}i^2})}{(e^{-\pi \frac{\sqrt{2}}{\sqrt{r}}\frac{-i \alpha}{\sqrt{2r}}i^2}-e^{\pi \frac{\sqrt{2}}{\sqrt{r}}\frac{-i \alpha}{\sqrt{2r}}i^2})} 
=e^{\frac{-\pi i \alpha \beta}{r}}\frac{(e^{\frac{-\pi i \alpha r}{r}}-e^{\frac{\pi i \alpha r}{r}})}{(e^{\frac{-\pi i \alpha}{r}}-e^{\frac{\pi i \alpha}{r}})} \\
&=q^{\alpha \beta}\frac{ \{r \alpha\}}{\{ \alpha\}}
=\frac{t_{V_{\alpha}}(\Phi_{V_{\beta},V_{\alpha}})}{t_{V_{\alpha}}(\Phi_{S_0,V_{\alpha}})},
\end{align*}
\begin{align*}
\text{qdim}[M_{1-k,j+1}^{\frac{-i \alpha}{\sqrt{2r}}}]&=e^{\pi k \sqrt{2r} \frac{-i \alpha}{\sqrt{2r}}}\frac{(e^{- \pi (j+1)\frac{\sqrt{2}}{\sqrt{r}} \frac{-i \alpha}{\sqrt{2r}} i^2}-e^{ \pi (j+1)\frac{\sqrt{2}}{\sqrt{r}} \frac{-i \alpha}{\sqrt{2r}} i^2})}{(e^{- \pi \frac{\sqrt{2}}{\sqrt{r}} \frac{-i \alpha}{\sqrt{2r}} i^2}-e^{ \pi \frac{\sqrt{2}}{\sqrt{r}} \frac{-i \alpha}{\sqrt{2r}} i^2})}\\
&=e^{- \pi k i \alpha}\frac{(e^{\frac{- \pi (j+1)i \alpha}{r}}-e^{\frac{\pi (j+1)i \alpha}{r}})}{(e^{\frac{- \pi i \alpha}{r}}-e^{ \frac{\pi i \alpha}{r}})}
=q^{rk \alpha}\frac{ \{(j+1) \alpha\}}{\{ \alpha \}}
=\frac{t_{V_{\alpha}}(\Phi_{S_j \otimes \Ck, V_{\alpha}})}{t_{V_{\alpha}}(\Phi_{S_0 \otimes V_{\alpha}})}
\end{align*}
and for $Re(\epsilon)<B_\epsilon^r$ together with $\epsilon \in \mathbb{S}(k,j+1+r(k+1))$
\begin{align*}
\text{qdim}[M^{\epsilon}_{1-k',i+1}]&=(-1)^{(j+kr+1-r)k'}\frac{\{(i+1)(j+1+r(k+1))\}}{\{j+1+r(k+1)\}} \\
&=(-1)^{i(k+1)+(j+kr+1-r)k'}\frac{\{(i+1)(j+1)\}}{\{j+1\}}=\frac{t_{P_j \otimes \Ck}(\Phi_{S_i\otimes \mathbb{C}_{k'r}^H,P_j \otimes \Ck} \circ x_{j,k})}{t_{P_j \otimes \Ck}(\Phi_{S_0,P_j \otimes \Ck} \circ x_{j,k})}.
\end{align*}
Finally in this region both $\text{\rm qdim}[F^{\epsilon}_{\frac{\beta + r -1}{\sqrt{2r}}}]$ and $\frac{t_{P_j \otimes \Ck}(\Phi_{V_\beta, P_j \otimes \Ck} \circ x_{j,k})}{t_{P_j \otimes \Ck}(\Phi_{S_0,P_j \otimes \Ck} \circ x_{j,k})}$ vanish.
\end{proof}

\begin{corollary}\label{cor:comp}
Let $\alpha \in \ddot{\mathbb{C}}$. Then the map $\varphi :V_{\alpha} \mapsto F_{\frac{\alpha+r-1}{\sqrt{2r}}}$, $S_i \otimes \Ck \mapsto M_{1-k,i+1}$ extended linearly over direct sums for $\alpha \in \ddot{\mathbb{C}}$ is a morphism up to equality of characters.
\end{corollary}

\begin{proof}
This can be directly verified via computation. It however also follows since the Hopf links as well as the asymptotic dimensions uniquely specify the (conjectured) tensor ring up to equality of characters; see Theorem 28 of \cite{CM} and \cite{CGP}. 
\end{proof}

\section{Towards braided tensor category structure on $\mathcal{M}(r)$-Mod}

\subsection{$C_1$-cofiniteness}

In this section we obtain sufficient conditions for the existence of a braided tensor category (and more) structure on a suitable category of 
$\mathcal{M}(r)$-modules.
In other words, we discuss applicability of the Huang-Lepowsky-Zhang tensor product theory \cite{HLZ} to $\mathcal{M}(r)$-Mod.
We also comment on the $r=2$ case due to recent rigorous derivation of the fusion ring for $\mathcal{M}(2)$ \cite{AM3}. For $r \geq 3$, this ring 
is known only conjecturally \cite{CM}. We believe that the powerful technique from \cite{AM3} can be extended to all $r$.

Recall that a $V$-module $M$ is $C_1$-cofinite if the subspace 
$$C_1(M):= \langle v_{-1} m :   v \in V , {\rm wt}(v)>0, m \in M  \rangle$$
 is of finite-codimension in $M$. 
We shall focus on the category of finite-length $C_1$-cofinite modules, denoted by $\mathcal{M}(r)$-Mod. Observe that this category is closed 
under finite products/coproducts and taking quotients. So it has an abelian category structure and thus we are in the position to apply parts of \cite{HLZ} theory.
The next result is needed in the proof of the main result.

\begin{lemma} \label{c1-vir} Let $L(c,h)$ be a non-generic (i.e. not isomorphic to the Verma module) Virasoro module for the Virasoro vertex operator algebra
$L(c,0)$. Then $L(c,h)$ is $C_1$-cofinite, viewed as an $L(c,0)$-module.
\end{lemma}
\begin{proof} By definition, every non-generic Verma module admits a singular vector $0 \neq w \in M(c,h)$ of weight $n$ (depending on $c$ and $h$). 
From the structure of singular vectors for the Virasoro algebra  (see Chapter 5, \cite{IK}) we get a decomposition 
$$w=L(-1)^n v_{c,h}+ w' v_{c,h}$$
where $v_{c,h}$ is the highest weight vector in $L(c,h)$ and $w' \in U({\rm Vir}_{<0})$, which is lower in the filtration than $L(-1)^n$. 
Clearly, we have $w' v_{c,h} \in C_1(L(c,h))$. 
We claim that 
$$L(c,h)=C_1(L(c,h)) + \sum_{i=0}^{n-1} L(-1)^i v_{c,h}.$$
Denote the right hand-side by $U$. Clearly we only have to prove that $L(-1)^i v_{c,h} \in U$ for $i \geq 0$. This follows immediately 
because of $L(-1)(C_1(L(c,h))) \subset C_1(L(c,h))$.
\end{proof}

\begin{theorem} \label{typ-atyp} All irreducible $\mathcal{M}(r)$-modules are $C_1$-cofinite.
\end{theorem}

\begin{proof} 
We first consider atypical modules. It is known that 
$$M_{t,s}=\bigoplus_{n=0}^\infty L^{\rm Vir}(c_{r,1}, h^{t,s}_n),$$
where $L^{Vir}(c_{r,1}, h^{t,s}_n)$ are certain irreducible Virasoro modules of central charge $c_{r,1}$; for explicit formulas see for instance \cite{AM1,AM2, CM}. 
These singlet modules can be realized inside the lattice vertex algebra modules by using the short screening $Q=e^{\alpha}_0$ acting on special highest weight 
vectors $e^{\gamma}$ inside the generalized lattice vertex algebra $V_{L'}$. More precisely, we have
$$M_{t,s}=\bigoplus_{n=0}^\infty U({\rm Vir}_{<0}).v^{(n)}=\bigoplus_{n=0}^\infty U({\rm Vir}_{<0}).Q^n e^{\beta_{t,s}-n \alpha}, \ \ \beta_{t,s} \in L'.$$

{\em Claim:} For every $n \geq 1$, there is a nonzero constant $C$ and $j \leq -1$ (depending on $n$) such that 
$$H_j v^{(n)}=C v^{(n+1)}+ w,$$  
where $$w \in V_n:=\displaystyle{\bigoplus_{i=0}^n }U({\rm Vir}_{<0}).v^{(i)}.$$
This claim follows along the lines of Lemma 5.3 in \cite{AM1}.

It is now sufficient to show that 
\begin{equation} \label{ind-proof}
M_{t,s}=C_1(M_{t,s})+ \bigoplus_{i=0}^{k_{t,s}} L(-1)^i v^{(0)},
\end{equation}
where $k_{t,s} \in \mathbb{N}$ depend only on $L(c_{r,1},h^{t,s}_0)$. Notice that this follows from Lemma \ref{c1-vir} once we prove
that
\begin{equation} \label{span}
H_{-i_1} \cdots H_{-i_k} L_{-j_1} \cdots L_{-j_\ell}v^{(0)},
\end{equation}
$i_m \geq 1$, $1 \leq m \leq k$,  $j_n \geq 1$, $ 1 \leq n \leq \ell$, is a spanning set for $M_{t,s}$. 
Let $V_n$ be defined as above. We shall prove by induction on $n \geq 0$,  that  $V_n$ is spanned by vectors in (\ref{span}). For $n=0$, this follows 
from Lemma \ref{c1-vir}. Now the above claim together with  the inductive hypothesis imply that 
$V_{n+1}$ is spanned by vectors of the form
$$L_{-k_1} \cdots L_{-k_p} H_{-i_1} \cdots H_{-i_k} L_{-j_1} \cdots L_{-j_\ell}v^{(0)}.$$
By using the bracket relations among $L_i$ and $H_j$, we can move all Virasoro generators that are on the left across the $H$-generators to the 
right. The resulting vector is clearly inside the span of vectors in (\ref{span}).

Now we switch to typical modules $F_{\lambda}$, $\lambda \notin L'$. At first we let $\lambda$ to be arbitrary. 
We already know that all atypicals are $C_1$-cofinite. In particular, this family includes infinitely many Fock spaces $F_\lambda$, $\lambda \in L$. 
 We will need the following two known facts. Fact 1. If $A=\mathbb{C} \setminus B$, where $B$ is a countable set, then any complex number can be 
 written as a finite sum of elements in $A$. Fact 2 (Miyamoto \cite{Mi1}) Suppose that $M$ and $N$ are $C_1$-cofinite and 
 $\mathcal{Y} \in { W \choose M \ N  }$ is surjective, then $W$ is also $C_1$-cofinite. We shall apply the latter for $M=F_{\lambda}$, $N=F_{\lambda'}$ 
 and $W=F_{\lambda'+\lambda}$, with $\mathcal{Y}$ the obvious Heisenberg VOA intertwining operator $\mathcal{Y}$ acting among three modules.
 
 Now we proceed with the proof. Clearly, it is sufficient to show that $F_{\lambda}$ is $C_1$-cofinite for  $\lambda \in B$, where $B$  is as above.
 Let  $$F_\lambda=\bigoplus_{n=0}^\infty F_\lambda(n),$$ 
 graded space  decomposition of $F_\lambda$. Then 
 $F_{\lambda}$ being $C_1$-cofinite means that 
 $$F_\lambda(n) \subseteq C_1(F_{\lambda}), \ \ n \geq k,$$
 for some fixed $k \in \mathbb{N}$. Consider $$C_1{F_\lambda}(n):=C_1(F_{\lambda}) \cap F_{\lambda}(n).$$
Because $C_1(F_\lambda)(n)$ depends polynomially on $\lambda$ (due to the fact that it picks up $\lambda$ when we act with the Heisenberg generator $\varphi(0)$ on $v_\lambda$) there will be only finitely many $\lambda$ values for which 
$dim(C_1(F_\lambda)(n))$ will drop ({\em non-generic} values) and for all other values  $dim(C_1(F_\lambda)(n))$
will be constant ({\em generic} values). We know that there is $\lambda$ for which $F_\lambda$ is $C_1$-cofinite and 
thus $dim(C_1(F_\lambda)(n))=p(n)$, here $p(n)$ is the number of partitions of $n$, for all $n \geq k_\lambda$. For every $n \geq k_{\lambda}$, denote 
by $B_i$ the set of all non-generic $\lambda$-values for $dim(C_1(F_\lambda)(i))<p(i)$ (this set is always finite). We let 
$$A:= \mathbb{C} \setminus \bigcup_{i=k_\lambda}^\infty B_i$$
and $B= \bigcup\limits_{i=k_\lambda}^\infty B_i$ is a desired countable set. The proof follows.
\end{proof}

\begin{corollary} All finite-length $\mathcal{M}(r)$-modules are $C_1$-cofinite. 
\end{corollary}

\begin{remark} \label{Vir-Heis} 
It is not hard to see that 
${F}_{\lambda}$ is not $C_1$-cofinite if viewed as a module for the Virasoro vertex operator algebra.
\end{remark}

\subsection{Fusion rules}

Formula (\ref{CM-fusion}), when specialized to $p=2$,  gives the following conjectural relations in the Grothendieck ring
(here $\lambda, \mu \notin L'$):

\begin{eqnarray}
\label{5} && [F_\lambda] \times [F_\mu] = [F_{\lambda+\mu}] +   [F_{\lambda+\mu+\alpha_-}], \qquad\quad \\
\label{6} && [M_{t ,1}]\times  [F_\mu] = [F_{\mu+\alpha_{r,1}}], \\
\label{7} && [M_{t ,1}]\times  [M_{t',1}^\epsilon] =[M_{t+t'-1,1}], \\
\label{8} && [M_{t ,2}^\epsilon]\times  [M_{t',1}^\epsilon] =[M_{t+t'-1,2 }^\epsilon], \\
\label{9} && [M_{t ,2}^\epsilon]\times  [M_{t',2}^\epsilon] =2 [M_{t+t'-1,1}^\epsilon]+ [M_{t+t'-2,1}^\epsilon]+[M_{t+t',1}^\epsilon] \\
 &&=[F_{\alpha_{t+t'-1,1}}]+[F_{\alpha_{t+t'-2,1}}]. \nonumber
\end{eqnarray}

Next we discuss rigorous results pertaining to fusion rules of irreducible $\mathcal{M}(2)$-modules.
As in \cite{AM3}, for a triple of equivalence classes of irreducible $\mathcal{M}(2)$-modules, $[M]$, $[N]$ and $[K]$, where $M$, $N$ and $K$ are representatives of classes, respectively, we define
\begin{equation} \label{fusion}
[M] \times [N] :=\sum_{[K] \in {\rm Irrep}} {\rm dim} \  I \ {K \choose M \ N } [K],
\end{equation}
where Irrep denotes the set of all equivalence classes of irreducible modules.
In \cite{AM3}, D. Adamovi\'c and the second author essentially proved the following result, thus verifying the correctness of fusion rules formulas obtained 
 conje\cite{CM} for $r=2$:
\begin{theorem} \label{am} We have:
\begin{itemize}[leftmargin=*,align=left]
\item[(i)] All typical modules $M_{t,1}$ are simple currents, in the sense that for a given irreducible module $N$ there is a unique irreducible module $M$ 
such that $I { M \choose M_{1,t} \ N }$ is nontrivial and one-dimensional. Moreover, under the product (\ref{fusion}), formula (\ref{5}) holds if $\lambda+\mu \notin L'$, and 
formulas (\ref{6})-(\ref{8}) also hold.
\item[(ii)] If $\lambda+\mu \in L'$, then there exists a logarithmic module $P$ such that in the Grothendieck group $[P]=[F_{\lambda+\mu}] +   [F_{\lambda+\mu+\alpha_-}]$
and $I \ { P \choose  F_{\lambda} \ F_{\mu}  } \neq 0$.
\item[(iii)] There exists a logarithmic module $Q$, with $[Q]=2 [M_{t+t'-1,1}^\epsilon]+ [M_{t+t'-2,1}^\epsilon]+[M_{t+t',1}^\epsilon]$ such that  
$I \ { Q \choose  M_{t ,2} \  M_{t',2} } \neq 0$.
\end{itemize}
\end{theorem}
Parts (i) and (ii) are already proven in \cite{AM3} in the setup of lattice vertex algebras. 
Part (iii) can be also proven by using \cite{AM3} (additional details will appear elsewhere). Alternatively, we can also argue as follows. As we know, the $r=2$ case corresponds to a subalgebra of the rank one symplectic fermions.  Logarithmic intertwining operator based on symplectic fermions modules was constructed in  \cite{Ru}, albeit in a slightly different setup. By virtue of restriction to the singlet algebra modules we obtain a family of (logarithmic) intertwining operators, leading to (iii).

\subsection{Tensor category structure}
According to [HLZ], in order for a suitably (sub)category of $V$-modules $\mathcal{C}$ to have the structure of a braided tensor category,  
it suffices that assumptions 10.1,12.1 and 12.2 in \cite{HLZ} hold.
Let ${C}$ denote the category of $\mathbb{N}$-gradable $C_1$-cofinite $\mathcal{M}(r)$-modules. 
By \cite{Mi1} any such module is logarithmic, that is its graded components decompose into generalized spaces with respect to the Virasoro generator $L(0)$.
Denote its full subcategory of finite-length $(1,r)$-modules by ${C}_{fin}$.  Now we find sufficient conditions for applicability 
of the HLZ-tensor product theory.

\vskip  5mm

Assumptions 10.1 (i)-10.1 (v) are satisfied in both categories. Assumptions 12.1 and 12.2 are also satisfied if every 
finitely generated $\mathbb{N}$-graded, generalized module is $C_1$-cofinite (see \cite{HLZ} for details), which we assume below.  It remains to analyze:

Assumption 10.1 (vi): For any object of $\mathcal{C}$, the (generalized) weights are real numbers and in addition there
exists $K \in \mathbb{Z}_+$ such that $(L(0)-L(0)_{ss})^K = 0$ on the generalized module. Notice that this condition 
holds in $C_{fin}$. Indeed, a finite-length module $M$ must satisfy this condition because of $K \leq \ell (M)$, where $\ell(M)$ is the length of $M$. This condition might not  hold in $C$; see however below.

Assumption 10.1 (vii): $\mathcal{C}$ is closed under images, under the contragredient functor, under taking finite direct
sums, and under $P(z)$-tensor products for some $z \in \mathbb{C}^\times$.
Because the dual of a finite length modules is of finite length, $C_{fin}$ is closed under the contragredient functor.
However, in the category $C$, it is a priori not clear that the dual of a module remains $C_1$-cofinite.  Finally, we have to check whether our categories are closed under the $P(z)$-tensor product. For the category $C$, this is implicitly verified in the paper of Miyamoto (see Main Theorem in \cite{Mi1}), albeit in a slightly different formulation. 
But for $C_{fin}$, it is not clear whether their tensor product (which exists) is still in $C_{fin}$. Based on what we already know about $C$ we believe that

\begin{theorem}\label{thm:M2} Assume that ${C}={C}_{fin}$, that is, every $C_1$-cofinite $\mathbb{N}$-graded module is of finite length and that every finitely generated generalized, $\mathbb{N}$-graded $\mathcal{M}(r)$-module is $C_1$-cofinite. Then, by virtue of \cite{HLZ}, the category ${C}$ can be equipped with a braided tensor structure. \end{theorem}


\begin{remark} \label{last} It is perhaps unclear why  the Heisenberg vertex algebra has not been much discussed in the context of tensor categories. For one, this vertex algebras is much easier to study.  The reason is that the category of $F_0$-modules is much less interesting compared to the singlet algebra and in addition does not produce any non-trivial quantum invariants. The category of modules for the Heisenberg vertex algebra $F_0$ that are diagonalizable under the zero mode subalgebra $\frak{h}$
is known to be semisimple. Intertwining operators, tensor product and associativity in the sense of \cite{HLZ} can be easily verified and all irreducible 
modules are clearly $C_1$-cofinite, see \cite{CKLR} for details. There is an enlargement of this category by inclusion of generalized (or logarithmic) $F_0$-modules \cite{Mil1}. 
(Logarithmic) intertwining operators of logarithmic $F_0$-modules are completely classified and explicitly constructed in \cite{Mil1} (see also \cite{Ru}). Now we consider the subcategory of finite-length $F_0$-modules. It is not hard to see that any finitely generated, generalized $\mathbb{N}$-graded  $F_0$-module is of finite-lenght. By combining results of \cite{Mil1, Ru} with \cite{HLZ}, we infer that this category is indeed braided. Again, this new category is not terribly interesting; for instance, as a tensor category, it is equivalent to the category of finite-dimensional $\frak{h}$-modules equipped with the 
usual tensor product. Of course, it would be desirable to work out complete details even in the Heisenberg case, but that is outside the scope of the present paper (see however \cite{Ru}).
\end{remark}

\hspace*{5mm}

\noindent Department of Mathematical and Statistical Sciences, University of Alberta,
Edmonton, Alberta  T6G 2G1, Canada. 
\emph{emails: creutzig@ualberta.ca, mrupert@hotmail.ca}

\hspace*{5mm}

\noindent Department of Mathematics and Statistics, SUNY-Albany, 1400 Washington Avenue, Albany, NY 12222, USA.
\emph{email: amilas@albany.edu}

\end{document}